\documentclass[oneside,english,reqno, openright, 11pt]{amsart}
\usepackage{ae,aecompl}

\usepackage[T1]{fontenc}
\usepackage[utf8]{luainputenc}
\usepackage{geometry}
\usepackage{xcolor}
\usepackage{pdfcolmk}
\usepackage{amsthm}
\usepackage{enumitem}
\usepackage[numbers,square]{natbib}
\usepackage{amstext}
\usepackage{amssymb}
\usepackage[all]{xy}
\usepackage{hyperref}
\usepackage{mathrsfs}
\usepackage{listings}
\usepackage{pdflscape}
\usepackage{verbatim}
\usepackage[toc,page]{appendix}
\PassOptionsToPackage{normalem}{ulem}
\usepackage{ulem}

\makeatletter

\numberwithin{equation}{section}
\numberwithin{figure}{section}
\theoremstyle{plain}

\ifx\thechapter\undefined
\newtheorem{theorem}{\protect\theoremname}
\else
\newtheorem{theorem}{\protect\theoremname}[chapter]
\fi
  \theoremstyle{definition}
  \newtheorem{defn}[theorem]{\protect\definitionname}
  \theoremstyle{plain}
  \newtheorem*{lem*}{\protect\lemmaname}
  \theoremstyle{plain}
  \newtheorem{lem}[theorem]{\protect\lemmaname}
  \theoremstyle{plain}
  
  \theoremstyle{plain}
  \newtheorem*{prop*}{\protect\propositionname}
  \theoremstyle{remark}
  \newtheorem{rem}[theorem]{\protect\remarkname}
  \theoremstyle{plain}
  \newtheorem{prop}[theorem]{\protect\propositionname}
  \theoremstyle{remark}
  
  \theoremstyle{definition}
  \newtheorem{ex}[theorem]{\protect\examplename}

\def\subsection{\@startsection{subsection}{3}%
  \z@{.5\linespacing\@plus.7\linespacing}{.1\linespacing}%
  {\normalfont\bfseries}}

\makeatother

\usepackage{babel}

\begin{document}

\providecommand{\claimname}{Claim}
\providecommand{\corollaryname}{Corollary}
\providecommand{\definitionname}{Definition}
\providecommand{\lemmaname}{Lemma}
\providecommand{\propositionname}{Proposition}
\providecommand{\remarkname}{Remark}
\providecommand{\examplename}{Example}
\providecommand{\theoremname}{Theorem}

\global\long\def\C{\mathbb{C}}
\global\long\def\N{\mathbb{N}}
\global\long\def\R{\mathbb{R}}
\global\long\def\Z{\mathbb{Z}}
\global\long\def\Q{\mathbb{Q}}
\global\long\def\A{\mathbb{A}}
\global\long\def\F{\mathbb{F}}
\global\long\def\P{\mathbb{P}}
\global\long\def\H{\mathbb{H}}
\global\long\def\ms#1{\mathscr{#1}}
\global\long\def\mb#1{\mathbb{#1}}
\global\long\def\mf#1{\mathfrak{#1}}
\global\long\def\mc#1{\mathcal{#1}}
\global\long\def\mr#1{\mathrm{#1}}
\global\long\def\dau{\partial}
\global\long\def\bdau{\overline{\dau}}
\global\long\def\sd{\mbox{d}}
\global\long\def\im{\mathrm{im}}
\global\long\def\cok{\mathrm{coker}}
\global\long\def\End{\mathrm{End}}
\global\long\def\so{\mathfrak{so}}
\global\long\def\Ad{\mathrm{Ad}}
\global\long\def\ad{\mathrm{ad}}
\global\long\def\stab{\mathrm{stab}}
\global\long\def\<{\left<}
\global\long\def\>{\right>}
\global\long\def\su{\mathfrak{su}}
\global\long\def\so{\mathfrak{so}}
\global\long\def\sp{\mathfrak{sp}}
\global\long\def\g{\mathfrak{g}}
\global\long\def\h{\mathfrak{h}}
\global\long\def\m{\mathfrak{m}}
\global\long\def\n{\mathfrak{n}}
\global\long\def\k{\mathfrak{k}}
\global\long\def\a{\mathfrak{a}}
\global\long\def\b{\mathfrak{b}}
\global\long\def\s{\mathfrak{s}}
\global\long\def\z{\mathfrak{z}}
\global\long\def\p{\mathfrak{p}}
\global\long\def\Span{\mbox{Span}}
\global\long\def\Der{\mbox{Der}}

\title{A new construction of naturally reductive spaces}
\author{Reinier Storm}

\maketitle

\begin{abstract}
A new construction of naturally reductive spaces is presented. This construction gives a large amount of new families of naturally reductive spaces. First the infinitesimal models of the new naturally reductive spaces are constructed. A concrete transitive group of isometries is given for the new spaces and also the naturally reductive structure with respect to this group is explicitly given.
\end{abstract}

\tableofcontents

\section{Introduction}
Naturally reductive spaces form the simplest class of Riemannian homogeneous spaces. They contain the symmetric spaces but also many other non-symmetric homogeneous spaces. The naturally reductive structure on a homogeneous Riemannian manifold is a metric connection $\nabla$ with totally skew-symmetric torsion $T$ and such that $\nabla T = \nabla R = 0$, where $R$ is the curvature tensor of $\nabla$. In particular naturally reductive spaces are simple examples of spaces with metric connection with totally skew-symmetric. In recent years there is an increasing interest in such structures together with parallel spinors (e.g. \cite{FriedrichIvanov2002}  and references therein).

The simple geometric and algebraic properties of naturally reductive spaces allows one to classify them in small dimensions. This has been done in \cite{TricerriVanhecke1983,KowalskiVanhecke1983,KowalskiVanhecke1985}   in dimension $3,~4,~5$ and more recently in dimension $6$ in \cite{AgricolaFerreiraFriedrich2015}. In \cite{AgricolaFerreiraFriedrich2015} there are also some interesting $G$-structures described on the naturally reductive spaces, where the naturally reductive connection is also a characteristic connection for the $G$-structure. In a forthcoming paper we investigate how the construction presented here relates to the classification problem of naturally reductive spaces.

Our construction produces many new examples of naturally reductive spaces. All of these spaces have a non-semisimple transvection algebra and are not normal homogeneous with its canonical naturally reductive structure. To our knowledge the only such construction is by Gordan in \cite{Gordon1985}. Here she constructs naturally reductive structures on 2-step nilpotent Lie groups. The construction presented here can be seen as a generalisation of the above mentioned construction of the 2-step nilpotent Lie groups. Our construction starts with three pieces of data. We take a naturally reductive space $M$ together with a Lie algebra $\k$ with an $\ad(\k)$-invariant metric on $\k$. The algebra $\k$ is a certain subalgebra of derivations of the transvection algebra of $M$. From this data we can construct a new naturally reductive space which is a homogeneous fiber bundle over $M$. If the naturally reductive space we start with is the symmetric space $\R^n$ and $\k\subset \so(n)$ is a subalgebra together with any $\ad(\k)$-invariant metric then we obtain exactly the 2-step nilpotent Lie groups with a naturally reductive structure from \cite{Gordon1985}. However we can start with any base space and a suitable subalgebra $\k$ and obtain many new examples of naturally reductive spaces which are not normally homogeneous with their canonical connection.
\\

In Section \ref{sec:preliminaries} we briefly discuss infinitesimal models of homogeneous spaces and the Nomizu construction (cf. \cite{Tricerri1992}).

In Section \ref{sec:the construction} we define the construction. We do this by defining a new pair of tensors $(T,R)$ form an original infinitesimal model $(T_0,R_0)$ on $(\m,g_0)$ of a naturally reductive space and a certain Lie subalgebra $\k$ of the transvection algebra of the model $(T_0,R_0)$ together with an $\ad(\k)$-invariant metric $B$ on $\k$. We prove that the new pair $(T,R)$ defines an infinitesimal model of a naturally reductive space on $(\k\oplus \m,g=B+g_0)$.

In Section \ref{sec:further investigation} we apply the Nomizu construction to our new infinitesimal models. For this we will not use the transvection algebra of the new model, but a Lie algebra which is known as the double extension, see \cite{MedinaRevoy1985}, of the transvection algebra of the base space by the Lie algebra $\k$. A condition is given when the new infinitesimal models $(T,R)$ are regular. For every constructed space an explicit transitive group of isometries is given and the naturally reductive structure, i.e. the left invariant objects $(g,T,R)$ are described on the Lie algebra of this explicit group of isometries. We will illustrate some of the new types of naturally reductive spaces with explicit examples.

\addtocontents{toc}{\protect\setcounter{tocdepth}{-1}}

\section*{Acknowledgement}
This paper is part of my PhD thesis supervised by Professor Ilka Agricola. I would like to thank her for introducing me to the subject of naturally reductive spaces and for her ongoing support. Also I would like to thank Yurii Nikonorov for reading a preliminary version and for interesting comments on it.

\addtocontents{toc}{\protect\setcounter{tocdepth}{1}}

\section{Preliminaries \label{sec:preliminaries}}
Let $(M=G/H,g)$ be a Riemannian homogeneous manifold. Let $\g$ and $\h$ be the Lie algebras of $G$ and $H$, respectively. Let 
\[
 \g=\mathfrak{h}\oplus \mathfrak{m}
\]
be some reductive decomposition. The reductive decomposition induces a left invariant connection on the principle $H$-bundle $G\to G/H$  called the canonical connection of the complement $\m$. Its horizontal distribution is defined by 
\[
T_g G \supset \mc{H}_{g} = dL_g(\m),
\]
where $L_g:G\to G$ is the left multiplication by $g\in G$. The tangent bundle of $M$ is the associated bundle $TM\cong G\times_{\Ad(H)} \m$. For $X\in \g$ let $\overline{X}$ denote the corresponding fundamental vector field:
\[
\overline{X}(p) := \left.\frac{d}{dt}\right|_{t=0} e^{tX} \cdot p \in T_pM.
\]
We will denote the chosen origin of our homogeneous space by $o$. Remember that $\m$ is canonically identified with the tangent space at the origin by
\begin{equation}
X \mapsto \overline{X}(o) \in T_oM. \label{eq:m = T_oM}
\end{equation}
The associated covariant derivative on $TM$, denoted $\nabla$, of the canonical connection has parallel torsion and curvature: $\nabla T=\nabla R=0$. Having a connection with parallel torsion and curvature characterizes locally homogeneous spaces:

\begin{theorem}[\cite{AmbroseSinger1958}]
A Riemannian manifold is locally homogeneous if and only if there exists a metric connection $\nabla$ with torsion $T$ and curvature $R$ such that 
\begin{equation}
\nabla T = \nabla R = 0. \label{eq:AS connection}
\end{equation}
\end{theorem}
A metric connection satisfying \eqref{eq:AS connection} is called an \emph{Ambrose-Singer connection}. If $M$ is in addition to the above theorem complete, then its universal cover is globally homogeneous. The torsion $T$ and curvature $R$ of an Ambrose-Singer connection are completely determined by their values at $T_oM\cong \m$. The pair of tensors $(T,R)$:
\begin{equation}
T: \Lambda^2 \m \to \m,\quad  R: \Lambda^2 \m \to \so(\m),\label{eq:(T,R)}
\end{equation}
satisfies
\begin{align}
&R(X,Y)\cdot T = R(X,Y)\cdot R=0 \label{eq:par T and R} \\
&\mf{S}^{X,Y,Z} R(X,Y)Z - \mf{S} T(T(X,Y),Z) = 0 \tag{B.1} \label{eq:B1}\\ 
&\mf{S}^{X,Y,Z} R(T(X,Y),Z) =0 \tag{B.2} \label{eq:B2},
\end{align}
where $\mf{S}^{X,Y,Z}$ denotes the cyclic sum over $X,Y$ and $Z$ and $\cdot$ denotes the natural action of $\so(\m)$ on tensors. The first equation encodes that $T$ and $R$ are parallel objects for $\nabla$ and under this condition the first and second Bianchi identity become equations \eqref{eq:B1} and \eqref{eq:B2}, respectively. A pair of tensors $(T,R)$ as in \eqref{eq:(T,R)} on a vector space $\m$ with a metric $g$ satisfying \eqref{eq:par T and R}, \eqref{eq:B1} and \eqref{eq:B2} is called an infinitesimal model on $(\m,g)$. From the infinitesimal model $(T,R)$ of a homogeneous space one can construct a homogeneous space with infinitesimal model $(T,R)$. This construction is known as the Nomizu construction which is discussed briefly below.

Let 
\[
\h := \{h\in \so(\m) : h\cdot T= h\cdot R=0\}.
\]
The Nomizu construction associates to every infinitesimal model a Lie algebra 
\begin{equation}\label{eq:Nomizu Lie algebra}
\g = \h \oplus \m,
\end{equation}
by defining the following Lie bracket: $A,B \in \h$, $X,Y \in \m$ 
\begin{equation}
[A + X,B + Y ] := [A,B]_{\so(\m)} - R(X, Y) + A(Y) - B(X) - T(X,Y).\label{eq:Nomuzi Lie bracket}
\end{equation}
This bracket satisfies the Jabobi identity if and only if $R$ and $T$ satisfy the equations \eqref{eq:par T and R}, \eqref{eq:B1} and \eqref{eq:B2}. Now we take the simply connected Lie group $G$ with Lie algebra $\g$ and let $H$ be the connected subgroup with Lie algebra $\h$. The infinitesimal model is called regular if $H$ is a closed subgroup of $G$. If this is the case then clearly the canonical connection on $G/H$ has the infinitesimal model $(T,R)$ we started with.
In \cite[Thm. 5.2]{Tricerri1992} it is proved that every infinitesimal model coming from a globally homogeneous Riemannian manifold is regular.

\begin{rem}\label{rem:regular inf model for some im(R) < k < h}
Let $(T,R)$ be an infinitesimal model on $(\m,g)$ with Lie algebra $\g=\h\oplus \m$ as in \eqref{eq:Nomizu Lie algebra}. Note that equation \eqref{eq:par T and R} implies that $\im(R)\subset \h$. Let $\mf p\subset \h$ be any subalgebra with $\im(R)\subset \mf p $. Then $\g':=\mf p\oplus \m\subset \g$ is also a subalgebra. Let $G'$ be a Lie group with Lie algebra $\g'$ and let $P\subset G'$ be the connected subgroup with Lie subalgebra $\mf p\subset \g'$. If $P\subset G'$ is closed then the canonical connection on $G'/P$ from the decomposition $\g':=\mf p\oplus \m$ has the infinitesimal model $(T,R)$. Hence by \cite[Thm. 5.2]{Tricerri1992} the model $(T,R)$ is regular. Conversely if $(T,R)$ is regular then $H\subset G$ is closed, where $G$ is the simply connected Lie group with Lie algebra $\g$ and $H$ is the connected subgroup with Lie subalgebra $\h$. Let $G'\subset G$ be a connected subgroup with Lie subalgebra $\g'$. Then $G'$ acts transitively on $G/H$. Hence the isotropy group $P$ is closed. Note that $P$ is a subgroup of $G'$ with Lie subalgebra $\mf p \subset \g'$. 
\end{rem}

We will call a Riemannian manifold $(M,g)$ \emph{naturally reductive} if there exists a transitive group $G$ of isometries with isotropy group $H$ and a reductive decomposition $\g = \h\oplus \m$ such that the canonical connection has totally skew symmetric torsion. Since the torsion of the canonical connection is given by
\begin{equation}
T(X,Y)_o=-[X,Y]_\mathfrak{m},
\end{equation}
the naturally reductive condition on the Lie algebra $\g$ is explicitly given by:
\[
g([X,Y]_\mathfrak{m},Z) = - g(Y,[X,Z]_\mathfrak{m}),\quad \forall X,Y,Z\in \mathfrak{m},
\]
where the metric on $\m$, which we also denote by $g$, comes from the linear isomorphism \eqref{eq:m = T_oM}.
From now on every homogeneous space will be naturally reductive. We use the metric to make the identification $\Lambda^2 \m\cong \so(\m)$. For naturally reductive spaces the curvature tensor $R:\Lambda^2\m \to \Lambda^2 \m $ is a symmetric map with respect to the Killing form of $\so(\m)$ and equation \eqref{eq:B2} holds automatically, see \cite{AgricolaFerreiraFriedrich2015}. Throughout this paper we will identify $\m$ with its dual $\m^*$ using the metric $g$. In this way we see $T$ as an element in $\Lambda^3\m$ and $R$ as an element in $\Lambda^2\m\odot \Lambda^2 \m$, where $\odot$ denotes the symmetric tensor product.

\begin{rem}
The first Bianchi identity is equivalent to (cf.\cite{AgricolaFerreiraFriedrich2015})
\[
R^{\Lambda^4}=2\sigma_T:= \sum_{i=1}^m (e_i\lrcorner T)\wedge (e_i\lrcorner T),
\]
where $R^{\Lambda^4}$ denotes the $4$-form component of the curvature tensor $R$. In other words $R^{\Lambda^4}=b(R)$, where $b$ is the Bianchi map:
\[
b(R)(X,Y,Z,V) = \frac{1}{3}(R(X,Y,Z,V)+R(Y,Z,X,V)+R(Z,X,Y,V)).
\]
\end{rem}

\section{The construction \label{sec:the construction}}
In this section we let $(M=G/H,g_0)$ be a naturally reductive manifold with respect to the canonical connection, $\nabla$, of a reductive decomposition $\g=\h\oplus \m$. Note that this implies that $\im(R_0)\subset \ad(\h)$, where $R_0$ is the curvature tensor of $\nabla$. Let $T_0\in \Lambda^3 \m$ be the torsion of $\nabla$. We define the following Lie algebra 
\begin{equation}\label{eq:s(g)}
\mf s(\g) :=\{f\in \Der(\g) : f(\h)=\{0\},~ f(\m)\subset \m,~ f|_{\m} \in \so(\m)\}.
\end{equation}
We will usually simply write $\mf s$ instead of $\mf s(\g)$.

\begin{rem}\label{rem:minimal g for s}
Let $\g=\h\oplus \m$ be as in equation \eqref{eq:Nomizu Lie algebra} for an infinitesimal model of a naturally reductive space. Let $\g'=\mf p\oplus \m$ with $\im(R)\subset \mf p\subset \h$ then $\s(\g)\subset \s(\g')$. In particular $\s(\g')$ is largest when $\mf p= \im(R)$. Note that $\im(R)$ is a subalgebra of $\h$ by equation \eqref{eq:par T and R}. For this reason we will often pick the reductive decomposition $\mf p\oplus \m$ with $\mf p =\im(R)$.
\end{rem}

Let $\mf k\subset \mf s$ be a subalgebra and let $\varphi:\mf k \to \so(\m)$ the natural faithful Lie algebra representation. Because of this faithful representation we know that $\k$ is a compact Lie algebra and thus $\k$ admits positive definite $\ad(\k)$-invariant metrics. Let $B$ be some $\ad(\k)$-invariant metric on $\k$. Later on we will have two copies of the Lie algebra $\k$. To keep notation consistent with the sequel we let $\n=\k$ be the other copy, even though at this moment this notation has no use. 

\begin{defn}\label{def:k-extension inf model}
Let $g=B+g_0$ be a metric on $\n\oplus \m$ with $B$ any $\ad(\k)$-invariant metric on $\n$. Let $k_{1},\dots,k_{l}$ be an orthonormal basis of $\k$ and denote by $n_1,\dots, n_l$ the corresponding basis of $\n$. We define a pair of tensors $(T,R)$,  $T\in \Lambda^3(\n\oplus \m)$ and $R:\Lambda^2(\n\oplus \m)\odot \Lambda^2(\n\oplus \m)$, by
\begin{equation}
T=T_{0}+\sum_{i=1}^{l}\varphi(k_i)\wedge n_i+2T_{\n},\label{eq: T}
\end{equation}
where 
\[
T_{\n} (X,Y,Z) = B([X,Y],Z),\quad \mbox{for}~X,Y,Z\in \n,
\]
and $[-,-]$ is the Lie bracket of $\n=\k$, and $\varphi(k_i)\in \so(\m)\cong \Lambda^2\m$ is identified with a 2-form on $\m$. For the following Lie algebra representation we use the notation:
\[
\psi = \ad\oplus \varphi: \k \to \so(\n\oplus \m),
\]
where $\ad:\k \to \so(\k) = \so(\n)$ is just the adjoint representation. The curvature tensor $R$ is defined as
\begin{equation}
R=R_{0}+\sum_{i=1}^{l} \psi(k_i)\odot \psi(k_i),\label{eq: R}
\end{equation}
We call the pair $(T,R)$ the $(\k,B)$-\emph{extension} of $(T_0,R_0)$.
\end{defn}

We will prove that $(T,R)$ defines an infinitesimal model of a naturally reductive space on $(\n\oplus \m,g)$. For this we have to show that $T$ and $R$ are invariant under $\im(R)$ and that the first Bianchi identity is satisfied. To prove this we first have a little algebraic lemma.

\begin{defn}
Let $(V,g)$ be a finite dimensional vector space with a positive definite metric $g$. Let $\alpha\in\Lambda^{p}V$ and $\beta\in\Lambda^{q}V$ then we define
a $(p+q-2)$-form by 
\[
\alpha \barwedge\beta=\sum_{i=1}^{n}(e_{i}\lrcorner\alpha)\wedge(e_{i}\lrcorner\beta),
\]
where $e_{1},\dots,e_{n}$ is an orthonormal basis of $V$.
\end{defn}

Note that the operation $\alpha\barwedge\beta$ is independent of the basis. One easily checks that:

\begin{lem}
\label{lem:square operation of 2-forms}
If $\alpha\in\Lambda^2 V \cong \so(n)$ and $\beta\in\Lambda^q V$ then $\alpha\barwedge\beta=\alpha\cdot\beta$,
where $\cdot$ denotes the standard action of $\so(n)$ on $\Lambda^q V$ as derivations. Furthermore if $\alpha,\beta\in\Lambda^2 V$ then $\alpha\cdot\beta=[\alpha,\beta]$.
\end{lem}

\begin{prop} \label{prop:T and R im(R) invariant}
Let $T_0$, $T$, $R_0$ and $R$ be as in Definition \ref{def:k-extension inf model}. Then the tensors $T_0,R_0,T$ and $R$ are $\im(R)$-invariant.
\end{prop} 
\begin{proof}
It suffices to show that $T_0,R_0,T$ and $R$ are invariant under $\mbox{im}(R_0)+\psi(\k)$, because
\[
\im(R) \subset \mbox{im}(R_0)+\psi(\k)\subset \so(\n\oplus \m).
\]
For every $k\in \k$ and $X,~Y,~Z\in \m$ we have
\begin{align*}
(\psi(k)\cdot T_0)(X,Y,Z) &= -T_0(\psi(k)(X),Y,Z)-T_0(X,\psi(k)(Y),Z)-T_0(X,Y,\psi(k)(Z)) \\
                    &= g([\psi(k)(X),Y],Z)+g([X,\psi(k)(Y)],Z)+g([X,Y],\psi(k)(Z)) \\
                    &= g([\psi(k)(X),Y],Z)+g([X,\psi(k)(Y)],Z)-g([\psi(k),[X,Y]],Z) = 0.
\end{align*}
Hence we get that $\psi(\k)$ stabilizes $T_0$. The invariants of $T_{\n}$ under $\psi(\k)$ is just the Jacobi identity of $\n=\k$. To see that the second term in \eqref{eq: T} is invariant under $\k$ we do a little computation. Let $k\in \k$
\[
\psi(k)\cdot \left( \sum_{i=1}^l \varphi(k_i)\wedge n_i \right)=\left( \sum_{i=1}^l [\varphi(k),\varphi(k_i)]\wedge n_i +\varphi(k_i)\wedge \ad(k)(n_i) \right).
\]
For the second term we have
\begin{align}
\sum_{i=1}^l \varphi(k_i)\wedge \ad(k)(n_i) &= \sum_{i,j=1}^l \varphi(k_i)\wedge B([k,n_i] ,n_j)n_j = \sum_{i,j=1}^l \varphi(B([k,n_i] ,n_j) k_i)\wedge n_j \nonumber \\
 &= \sum_{i,j=1}^l -\varphi(B([k,n_j],n_i) k_i)\wedge n_j 
 = \sum_{i,j=1}^l -\varphi(B([k,k_j] ,k_i) k_i)\wedge n_j\nonumber \\
 &= \sum_{j=1}^l -\varphi([k,k_j])\wedge n_j = \sum_{j=1}^l -[\varphi(k),\varphi(k_j)]\wedge n_j. \label{eq: term of T invariant}
\end{align}
Plugging these results back into the first line we see that 
\[
\psi(k)\cdot \left( \sum_{i=1}^l \varphi(k_i)\wedge n_i \right)=0.
\]
This shows that $T$ is invariant under $\psi(\k)$. We have $\im(R_0)\subset \so(\m)\subset \so(\n\oplus \m)$, so $\im(R_0)$ acts trivially on $\n$. This immediately tells us that $\im(R_0)$ leaves $T_\n$ invariant.  
The definition of $\s$ gives us that for every $k\in \k\subset \s$, $h\in \h$ and $m\in \m$ we have
\[
k([h,m]) = [k(h),m]+[h,k(m)]= [h,k(m)].
\]
This implies that $\varphi(k)$ commutes with every element of $\ad(\h)$. By Lemma \ref{lem:square operation of 2-forms} we immediately see that $\im(R_0)$ also leaves the second summand of \eqref{eq: T} invariant. This concludes that $T_0$ and $T$ are invariant under $\im(R)$. 

The tensor $R_0$ is invariant under $\psi(\k)$, since $\psi(\k)$ commutes with $\im(R_0)$. The same argument also tells us that $\sum_{i=1}^{l}\psi(k_i)\odot \psi(k_i)$ is invariant under $\im(R_0)$. Lastly, by a similar computation as \eqref{eq: term of T invariant} one can see that the tensor $\sum_{i=1}^{l}\psi(k_i)\odot \psi(k_i)$ is invariant under $\psi(\k)$. We conclude that $R_0$ and $R$ are $\im(R)$-invariant.
\end{proof}

\begin{prop}\label{prop:Bianchi identity}
The pair of tensors $(T,R)$ from Definition \ref{def:k-extension inf model} satisfies the first Bianchi identity.
\end{prop}
\begin{proof}
Let $e_{1},\dots,e_{n}$ be an orthonormal basis of $\m$. We will compute $\sigma_{T}$ using the following definition 
\[
\sigma_{T}=\frac{1}{2}\left(\sum_{p=1}^n(e_p\lrcorner T)\wedge(e_p\lrcorner T)+\sum_{i=1}^l (n_i\lrcorner T)\wedge(n_i\lrcorner T)\right).
\]
For $(e_p\lrcorner T)\wedge (e_p\lrcorner T)\equiv (e_i \lrcorner T)^{\wedge 2}$ we have
\begin{align*}
(e_i \lrcorner T)^{\wedge 2} &= (e_p \lrcorner T_0)^{\wedge 2} +2\sum_{i=1}^l (e_p\lrcorner T_0)\wedge\varphi(k_i)(e_p)\wedge n_i+\sum_{i,j=1}^l\varphi(k_i)(e_p)\wedge n_i\wedge\varphi(k_j)(e_p)\wedge n_j\\
 & = (e_p\lrcorner T_0)^{\wedge 2} +2\sum_{i=1}^{l}(e_p\lrcorner T_0)\wedge\varphi(k_i)(e_p)\wedge n_i-\sum_{i,j=1}^{l}\varphi(k_i)(e_p)\wedge\varphi(k_j)(e_p)\wedge n_i\wedge n_j.
\end{align*}
Now we sum these three summands over $p$. The first summand this
gives
\[
2\sigma_{T_0}=\sum_{p=1}^n(e_p\lrcorner T_0)\wedge(e_p \lrcorner T_0),
\]
For the second we obtain 
\begin{align*}
2\sum_{p=1}^n\sum_{i=1}^l(e_p\lrcorner T_0)\wedge\varphi(k_i)(e_p)\wedge n_i &=  2\sum_{i=1}^l(\varphi(k_i)\barwedge T_0)\wedge n_i\\
 & = 2\sum_{i=1}^l(\varphi(k_i)\cdot T_0)\wedge n_i=0,
\end{align*}
where we used Lemma \ref{lem:square operation of 2-forms} and the last equality follows from the fact that $\varphi(k_i)$ stabilizes $T_0$.
For the third summand we again use Lemma \ref{lem:square operation of 2-forms} and obtain:
\begin{align*}
-\sum_{p=1}^n\sum_{i,j=1}^l\varphi(k_i)(e_p)\wedge\varphi(k_j)(e_p)\wedge n_i\wedge n_j &=-\sum_{i,j=1}^l[\varphi(k_i),\varphi(k_j)]\wedge n_i\wedge n_j\\
 &= -2\sum_{i=1}^l \varphi(k_i)\wedge \ad(k_i),
\end{align*}
where $\ad(k_i)\in \so(\k)\cong \Lambda^2\n$ and the last equality follows from:
\begin{align*}
\left( \sum_{i=1}^l \varphi(k_i)\wedge \ad(k_i)\right)(n_p,n_q) &= 
\sum_{i=1}^l \varphi(B([k_i,n_p],n_q)k_i)\\
 &= \sum_{i=1}^l \varphi(B(k_i,[k_p,k_q])k_i)= \sum_{i=1}^l \varphi(B(k_i,[k_p,k_q])k_i)\\
 &= \varphi([k_p,k_q]) =  [\varphi(k_p),\varphi(k_q)]\\
 &= \frac{1}{2} \left(\sum_{i,j=1}^l[\varphi(k_i),\varphi(k_j)]\wedge n_i\wedge n_j\right)(n_p,n_q).
\end{align*}
The last term to compute for $\sigma_{T}$ is 
\begin{align*}
\sum_{i=1}^l(n_i\lrcorner T)\wedge (n_i\lrcorner T)& = \sum_{i=1}^l(\varphi(k_i)\wedge\varphi(k_i)+4\varphi(k_i)\wedge(n_i\lrcorner T_{\n}))+2\sigma_{2T_{\n}}\\
 &=  \sum_{i=1}^l(\varphi(k_i)\wedge\varphi(k_i)+4\varphi(k_i)\wedge(n_i\lrcorner T_{\n})) \\
 &=  \sum_{i=1}^l(\varphi(k_i)\wedge\varphi(k_i)+4\varphi(k_i)\wedge\ad(k_i),
\end{align*}
here we used the Jacobi identity for $\n$ in the form $\sigma_{T_{\n}}=0$: 
\begin{align}\label{eq:Jacobi in terms of sigma_T}
2\sigma_{T_\n}(X,Y,Z) &= \sum_{i=1}^l (n_i\lrcorner T)\wedge
(n_i\lrcorner T)(X,Y,Z) = \left(\sum_{i=1}^l \ad(k_i)\wedge \ad(k_i)\right) (X,Y,Z) \notag \\
 &= \mf S^{X,Y,Z}\sum_{i=1}^l B([k_i,X],Y) [k_i,Z]  = \mf S^{X,Y,Z}[[X,Y],Z]=0.
\end{align}
Summing all the terms we obtain that
\[
\sigma_T=\sigma_{T_0}+\frac{1}{2}\sum_{i=1}^{l}\left(2\varphi(k_i)\wedge\ad(k_i)+\varphi(k_i)\wedge\varphi(k_i)\right).
\]
Computing $R^{\Lambda^4}$ is a bit easier. We have
\begin{align*}
R^{\Lambda^4} &= R_0^{\Lambda^4} + \sum_{i=1}^l (\varphi(k_i)+\ad(k_i))\wedge (\varphi(k_i)+\ad(k_i)) \\
 &= R_0^{\Lambda^4}+\sum_{i=1}^l\varphi(k_i)\wedge\varphi(k_i)+\ad(k_i)\wedge\ad(k_i) + 2\varphi(k_i)\wedge\ad(k_i)\\
  &= R_0^{\Lambda^4}+\sum_{i=1}^l\varphi(k_i)\wedge\varphi(k_i) + 2\varphi(k_i)\wedge\ad(k_i).
\end{align*}
Here we used that $\sum_{i=1}^l \ad(k_i)\wedge \ad(k_i)=0$ by \eqref{eq:Jacobi in terms of sigma_T}. Hence we see that this torsion and curvature satisfy the Bianchi identity.
\end{proof}

We obtain one of our main results directly from Proposition \ref{prop:T and R im(R) invariant} and \ref{prop:Bianchi identity}:

\begin{theorem}
Let $(T_0,R_0)$ be an infinitesimal model of a naturally reductive space on $(\m,g_0)$. Any $(\k,B)$-extension $(T,R)$ defines an infinitesimal model of a naturally reductive space on $(\n\oplus \m,g = B + g_0)$. 
\end{theorem}

Note that there can be a multitude of different Lie algebras $\k\subset \s$ for a given infinitesimal model $(T_0,R_0)$, e.g. Example \ref{ex:su(3)/s1 times R4}. Also any $\ad(\k)$-invariant metric $B$ on $\k$ gives us a $(\k,B)$-extension. This means that the newly constructed infinitesimal models always come with a parameter family of naturally reductive structures.

The next thing we will do is apply the Nomizu construction to the new infinitesimal models we just constructed. 

\begin{defn}\label{def:g(k)}
Let $(T_0,R_0)$ be an infinitesimal model of a naturally reductive space on $(\m,g_0)$. Let $\h:=\im(R_0)$ and let $(T,R)$ be a $(\k,B)$-extension of $(T_0,R_0)$. We define the following vector space
\[
\g(\k):=\h\oplus \k \oplus \n\oplus \m
\]
together with $g:=B+g_0$ as metric on $\n\oplus \m$. In the proof of Proposition \ref{prop:T and R im(R) invariant} we proved that
\[
\ad\oplus \psi:\h\oplus \k \to \{h\in \so(\n\oplus \m) : h\cdot T= h\cdot R=0\}.
\]
We can define a Lie bracket by \eqref{eq:Nomuzi Lie bracket} for the pair $(T,R)$. 
This vector space $\g(\k)$ is a subalgebra of \eqref{eq:Nomizu Lie algebra} for the pair $(T,R)$ by Remark \ref{rem:regular inf model for some im(R) < k < h}. We will call $\h\oplus \k$ the \emph{isotropy subalgebra} of $\g(\k)$.
\end{defn}

\begin{rem}
The constructed Lie algebra $\g(\k)$ is known as the double extension of $\g$ by $\k$. This construction is used in \cite{MedinaRevoy1985} to describe the set of all Lie algebras which possess an invariant non-degenerate bilinear form.
\end{rem}

Whenever we will construct a reductive decomposition $\g(\k)=\h\oplus \k\oplus \n\oplus \m$ from $\g=\h\oplus \m$ then we will call $\g=\h\oplus \m$ the \emph{base space} and $\g(\k)=\h\oplus \k\oplus \n\oplus \m$ the \emph{total space}. In the next section we briefly discuss that if both the base space and the total space are regular then the total space is a  homogeneous fiber bundles over the base space and the fiber directions are $\n$. We would like to point out that it is also possible to start from a locally homogeneous space $(T_0,R_0)$ which is not globally homogeneous and still obtain a globally homogeneous space with infinitesimal model $(T,R)$.

\section{Further investigation of the spaces $\g(\k)$ \label{sec:further investigation}}
In this section we will investigate when the newly constructed infinitesimal models $(T,R)$ are regular. For particular base spaces we will explicitly give a transitive group of isometries for the infinitesimal models $(T,R)$ and describe the naturally reductive structure on its Lie algebra. From now on $G(\k)$ will denote the simply connected Lie group with Lie algebra $\g(\k)$ and $H(\k)$ will be the connected subgroup with subalgebra $\h(\k)$. By Remark \ref{rem:regular inf model for some im(R) < k < h} we have that $H(\k)$ is closed is $G(\k)$ if and only if the infinitesimal model $(T,R)$ is regular.

The first thing we want to point out is that the diagonal \[
\Delta \k \subset \k\oplus \k \cong \k \oplus \n\subset \g(\k)
\]
is a non-trivial Abelian ideal of $\g(\k)$. In particular $\g(\k)$ is never semisimple. We will denote $\Delta \k\subset \k\oplus \n$ by $\a$. 

\begin{lem} \label{lem:ideals in g(k)}
Let $\g(\k)$ be the Lie algebra from Definition \ref{def:g(k)}, then the following hold:
\begin{enumerate}[label = \roman*)]
 \item $\a$ commutes with $\h\oplus \m$,
 \item the linear subspace $\mf a\subset \g(\k)$ is an Abelian ideal,
 \item $\mf l:= \h\oplus \mf a\oplus \m$ is an ideal,
 \item $\g(\k)\cong \k \ltimes (\h \oplus \mf a \oplus \m)$.
\end{enumerate}
\end{lem}
\begin{proof}
\emph{i)} Let $k+n\in \mf a$ and $m\in \m$ then
\[
[k+n,m] = \varphi(k)m - T(n,m) -R(n,m) = \varphi(k)m - T(n,m) = \varphi(k)m - \varphi(k)m = 0.
\]
We used that $R(n,m)=0$ for $n\in\n$ and $m\in \m$. This follows directly from the symmetries of $R$ and the $(\h\oplus \k)$-invariance of the direct sum $\n\oplus \m$. 
Let $h\in \h$ then we already saw in Proposition \ref{prop:T and R im(R) invariant} that $h$ commutes with $\k$ and that it acts trivially on $\n$. In particular $h$ will commutes with $\mf a$.

\emph{ii)} Let $n+k\in \a $ and $n'+k'\in \a$. We have that 
\begin{align*}
[n',n+k] &= [n',n] + [n',k] = -2[n',n]_\n - \sum_{i=1}^l B([k_i,n'],n)k_i + [n',k] \\
 &= - [n',n]_\n -\sum_{i=1}^l B([k_i,n'],n)k_i = - [n',n]_\n -\sum_{i=1}^l B(k_i,[k',k])k_i \\
 &= -[n',n]_\n - [k',k]\in \a,
\end{align*}
where $[-,-]_\n$ denotes the Lie bracket in $\n$. Furthermore we have $[k',n+k]= [n',n]_\n + [k',k]\in \a$. In particular we see that $[n'+k',n+k]=0$, thus $\a$ is Abelian. Moreover we have that $[\k\oplus \n,\a]\subset \a$ and together with \emph{i)} we see that $\a $ is an ideal.

\emph{iii)} We already know that $[\h\oplus \k,\mf l]\subset \mf l$. For $n\in \n$ and $m\in \m$ we have by \emph{i)} that $[n,m] = [k,m]\in \m$, where $k\in \k$ is such that $n+k\in\mf a$. This gives us that $[\h\oplus \k\oplus \n,\mf l]\subset \mf l$. The only remaining thing to check is that for $m_1,m_2\in \m$ also $[m_1,m_2]\in \h\oplus \mf a\oplus \m$. This is equivalent to $[m_1,m_2]_{\k\oplus \n} \in \mf a$. A short computation gives
\begin{align*}
[m_1,m_2]_{\k\oplus \n} &= -\sum_{i=1}^l \varphi(k_i)(m_1, m_2) n_i + \psi(k_i)(m_1, m_2) k_i\\
 &=-\sum_{i=1}^l \varphi(k_i)(m_1, m_2)(n_i+ k_i)\in \mf a.
\end{align*}

\emph{iv)} It only remains to note that $\k$ is a subalgebra of $\g(\k)=\k\oplus (\h\oplus \mf a\oplus \m)$.
\end{proof}

\begin{rem}
It was pointed out to us by Y. Nikonorov that the vectors in $\a$ constitute Killing vectors of constant length on $G(\k)/H(\k)$ as is proven in \cite{Nikonorov2013}. Note that the Lie algebra $\g(\k)$ we consider doesn't have to be the full Lie algebra of Killing fields as is the case in \cite{Nikonorov2013}. However one easily checks that the proves remain valid.
\end{rem}

For now we assume that the newly constructed infinitesimal model $(T,R)$ from an infinitesimal model $(T_0,R_0)$ is regular. As mentioned before this is equivalent to the subgroup $H(\k)\subset G(\k)$ being closed. Furthermore we assume that the base space $G/H$ is also regular and simply connected. We show that $G(\k)/H(\k)$ is a homogeneous fiber bundle over $G/H$ and that the bundle map $G(\k)/H(\k) \to G/H$ is a Riemannian submersion. From Lemma \ref{lem:ideals in g(k)} we know that $\mf r:= \h\oplus \k \oplus \a = \h\oplus \k \oplus \n$ is a subalgebra of $\g(\k)$ and that $\mf r \oplus \m$ is a reductive decomposition. Now we have that $\ad:\mf r \to \so(\m)$ maps into $\{h\in \so(\m):h \cdot T_0 = 0,~h\cdot R_0 =0\}$ by Proposition \ref{prop:T and R im(R) invariant}. Let $\mf q$ denote the kernel of this map. Then we have that $\mf q\subset \mf r\oplus \m$ is an ideal and that $(\mf r\oplus \m)/\mf q$ is naturally identified with a subalgebra of the isometry algebra $\mf{isom}(G/H)$ of $G/H$. Note that $\a \subset \mf q$. This quotient map from $\g(\k)=\mf r\oplus \m$ into to $\mf{isom}(G/H)$ lifts to $G(\k)$, because $G(\k)$ is simply connected.
Hence the group $G(\k)$ acts on $G/H$ by isometries. The stabilizer group of the origin of $G/H$ is a closed subgroup $R$ of $G(\k)$. We readily see that $\mbox{Lie}(R) = \mf r$. This means that we have a homogeneous fiber bundle which is a Riemannian submersion 
\[
R/H(\k) \mbox{ --- } G(\k)/H(\k) \to G(\k)/R \cong G/H.
\]
The fibers $R/H(\k)$ are connected by the long homotopy exact sequence and are described by the reductive decomposition 
\[
\h\oplus \k\oplus \n,
\]
with $\h\oplus \k$ the isotropy algebra. The canonical connection of this reductive decomposition is clearly a naturally reductive connection. Let $A$ be the connected subgroup of $G(\k)$ with Lie algebra $\a$. Note that $A\subset R$ acts already transitively and by isometries on the fibers. Since $A$ is an Abelian Lie group this means that the universal cover of $R/H(\k)$ is isomorphic to the symmetric space $\R^l$. The torsion $T_f$ and curvature $R_f$ of the naturally reductive connection on the fiber $\R^l$ are given by 
\[
T_f = 2T_\n \quad \mbox{and}\quad R_f = \sum_{i=1}^l \ad(k_i)\odot \ad(k_i),
\]
where $T_\n$ is as in Definition \ref{def:k-extension inf model}.

Now we will investigate for certain types of base spaces what kind of naturally reductive spaces our construction gives. For every constructed space we will give an explicit transitive group of isometries and we will describe the naturally reductive structure, i.e. the left invariant objects $(g,T,R)$ on the Lie algebra of this explicit group of isometries.

\subsection{Base space with $\g$ semisimple \label{subsec:g semisimple}}
In this section we will assume that $\g$ is a semisimple Lie algebra. Let $(M,g_0)=(G/H,g_0)$ be a naturally reductive space with respect to the canonical connection of a reductive decomposition $\g=\h\oplus \m$ with infinitesimal model $(T_0,R_0)$. Furthermore we assume that $G$ and $G/H$ are simply connected. By Remark \ref{rem:minimal g for s} we can without loss of generality assume that 
\[
\im(R_0)=\ad(\h)\subset \so(\m).
\]
Note that this means that for the reductive decomposition $\g=\h\oplus \m$ we have that $\g = [\m,\m]+\m $ and then a result by Kostant (\cite{Kostant1956}, see also \cite{D'AtriZiller1979}) tells us that there exists a unique $\ad(\g)$-invariant metric on $\g$ whose restriction to $\m$ is the naturally reductive metric $g_0$ and its restriction to $\h$ is non-degenerate. Moreover for this metric $\h$ and $\m$ are orthogonal to one another. We will denote the $\ad(\g)$-invariant metric on $\g$ by $\overline{g}$. Note that the metric $\overline{g}$ can have signature.

\begin{rem}
We would like to point out that there are many naturally reductive spaces which satisfy the above requirements. We can take any compact simple Lie group $G$ with any closed subgroup $H$ and then $\g = \h \oplus \h^\perp = \h \oplus \m$ is a reductive decomposition, where the orthogonal complement is taken with respect to a multiple of the Killing form. In particular these spaces are normal homogeneous. The canonical connection is naturally reductive with respect to the Killing form on $\g$ restricted to $\m$, moreover $\im(R_0)=\ad(\h)$. Also for $G$ compact and semisimple there are many examples. The difference is that $\im(R_0)=\ad(\h)$ doesn't hold automatically for every subgroup $H$. Also for $G$ non-compact and semisimple it is not hard to construct examples.
\end{rem}

Since $\g$ is semisimple all derivations of $\g$ are inner derivations. We see that the Lie algebra $\mf s$ is given by
\[
\mf s = \mf z \oplus \mf p,
\]
where $\mf z$ is the center of $\h$ and $\mf p$ are all vectors on which $\h$ acts as zero:
\[
\mf p := \{m\in \m : [h,m]=0,~ \forall h\in \h\}.
\]
Let $\k\subset \s$ be a subalgebra with an $\ad(\k)$-invariant metric $B$. We will denote the corresponding subalgebra of $\k$ in $\g$ by $\b$. Let $\b_\p:= \b \cap \p$ and $\b_\z := \b_\p^\perp$, where the orthogonal complement is taken in $\mf b$ with respect to $B$. Note that $\b = \b_\z \oplus \b_\p$ and that both $\b_\z$ and $\b_\p$ are ideals in $\b$. We denote the corresponding decompositions of $\n$, $\k$ and $\a$ by $\n=\n_\z\oplus \n_\p$, $\k=\k_\z\oplus \k_\p$ and $\a = \a_\z \oplus \a_\p$, respectively. Let $b_1,\dots ,b_l$ be an orthonormal basis of $\mf b$ with respect to $B$. Denote the corresponding basis of $\n$ by $n_1,\dots ,n_l$ and that of $\k$ by $k_1,\dots ,k_l$.
Define the following linear map
\[
a:\g \to \g(\k);~ a(x) := \sum_{i=1}^{l} \overline{g}(x,b_i) (n_i+k_i).
\]
With the help of the following lemma we will see that $\g$ is a subalgebra of $\g(\k)$. It will always be clear from the context whether an element $x\in \m$ or $x\in \h$ belongs to $\g$ or $\g(\k)$.
\begin{lem}\label{lem:Lie algebra hom f}
The linear map $f:\g \to \g(\k)$ defined by 
\[
f(x) = x - a(x)
\]
is an injective Lie algebra homomorphism.
\end{lem}
\begin{proof}
In the following we will denote the Lie bracket on $\g(\k)$ by $[-,-]$ and the Lie bracket on $\g$ by $[-,-]_{\g}$. Let $h\in \h\subset\g$ and $x\in \g$. Using that $a(x)\in \mf a$ for all $x\in \g$ we get by Lemma \ref{lem:ideals in g(k)} that
\[
[f(h),f(x)] = [h-a(h),x-a(x)] = [h,x] = [h,x]_\g = f([h,x]_\g),
\]
where the last equality follows because $[h,x]_\g\in \mf p^\perp\subset \m$. It remains to check that for all $m_1,m_2\in \m$ the following holds $f([m_1,m_2]_{\g}) = [f(m_1),f(m_2)]$. We compute the right-hand-side of this equation:
\begin{align*}
[f(m_1),f(m_2)]&=[m_1 - a(m_1),m_2 - a(m_2)]  = [m_1,m_2]\\
               &= -T(m_1, m_2) - R(m_1, m_2) \\
 &= -T_0(m_1, m_2)-R_0(m_1, m_2) - \sum_{i=1}^{l} \varphi(k_i) (m_1, m_2)(n_i+k_i)\\
  	      &= [m_1,m_2]_{\g} - \sum_{i=1}^{l} \overline{g}([b_i,m_1]_{\g},m_2)(n_i+k_i)\\
  	      &=[m_1,m_2]_{\g} - \sum_{i=1}^{l} \overline{g}([m_1,m_2]_{\g},b_i)(n_i+k_i)\\
  	      &= f([m_1,m_2]_{\g}).
\end{align*}
\end{proof}

Note that the ideal $\mf l = \h\oplus \mf a \oplus \n$ from Lemma \ref{lem:ideals in g(k)} \emph{iv)} is equal to the direct sum of Lie algebras $\mf l:=f(\g) \oplus \mf a$. Also the projection of
\[
f(\g) \oplus \a_\p\subset \g(\k)
\]
along $\h\oplus \k$ onto $\n\oplus \m$ is surjective. We have an injective Lie algebra homomorphism
\[
\phi:=f\oplus i:\g \oplus \a_\p \to \g(\k),
\]
where $i:\a_\p\to \g(\k)$ is the inclusion. Let $\Phi:G\times \R^{l_\p}\to G(\k)$ be the induced Lie group homomorphism on the simply connected Lie group $G\times \R^{l_\p}$, where $l_\p = \dim(\a_\p)$. Suppose that $H(\k)\subset G(\k)$ is closed, i.e. $(T,R)$ is regular. Then $G\times \R^{l_\p}$ acts transitively on $G(\k)/H(\k)$ through the map $\Phi$.   
 
Let $\h_0:= \ker(a|_\h)$. This is an ideal in $\h$, because for  $h_0\in \h_0$ and $h\in \h$ 
\[
a([h_0,h])=\sum_{i=1}^{l} \overline{g}([h_0,h],b_i) =\sum_{i=1}^{l} \overline{g}(h_0,[h,b_i])=0.
\]
Let $\h_1\subset \h$ a complementary ideal, this exists because $\h$ is compact. Let $b_1^{\z},\dots ,b_{l_\z}^\z$ be an orthonormal basis of $\b_\z$ with respect to $B$. Let $h_1,\dots ,h_{l_\z}$ be elements of $\h_1$ such that $\overline{g}(h_i,b_j^{\z})=\delta_{ij}$, these exist because $\overline{g}_{\h\times \h}$ is non-degenerate. The isotropy algebra for the action of $G\times \R^{l_\p}$ is given by
\[
\phi^{-1}(\h\oplus \k) =  \h_0.
\]
Hence if the infinitesimal model $(T,R)$ is regular then $H_0$, the connected subgroup of $G$ with Lie algebra $\h_0$, has to be a closed subgroup. Conversely if $H_0$ is closed then the infinitesimal model $(T,R)$ is regular by \cite{Tricerri1992}. The homogeneous space $G(\k)/H(\k)$ is  isomorphic to
\begin{equation} 
G/H_0 \times \R^{l_\p}.\label{eq:total space isomorphic to}
\end{equation}
We will now describe the naturally reductive structure directly on the reductive decomposition associated to \eqref{eq:total space isomorphic to}:
\[
\h_0 \oplus \h_1 \oplus \m \oplus \a_\p\quad \mbox{and}\quad \h_0 \mbox{  is the isotropy algebra} .
\]
We have an orthonormal basis $n_1^\z,\dots,n_{l_\z}^\z$ of $\n_\z$ and $n_1^\p,\dots ,n_{l_\p}^\p$ of $\n_\p$. Let $m_1,\dots,m_n$ be an orthonormal basis of $\m$ with respect to $g_0$. We know the formula for the torsion and curvature in this basis. Hence all we have to do is to find a basis $f_1^\z,\dots,f_{l_\z}^\z,f_1^\p,\dots,f_{l_\p}^\p, e_1,\dots , e_n$ of 
\[
\h_1 \oplus \m \oplus \a_\p,
\]
such that for the fundamental vector fields the following holds: 
\begin{align}
\overline{f_i^\z}(o) = \overline{n_i^\z}(o),\quad \overline{f_i^\p}(o) = \overline{n_i^\p}(o)\quad \mbox{and}\quad \overline{e_j}(o)=\overline{m_j}(o), \label{eq:new basis at origin}
\end{align}
where $o$ is the chosen origin. Remember that $G\times \R^{l_\p}$ acts on $G(\k)/H(\k)$ through the map $\Phi$. Such a basis will give an orthonormal basis of the tangent space at the origin and thus describe the left invariant metric. Also in this basis the formula for the torsion and curvature of the naturally reductive connection are as in Definition \ref{def:k-extension inf model}.

For $i=1,\dots ,l_\p$ we set
\[
f_i^\p := n_i^\p+k_i^\p \in \a_\p.
\]
For $i=1,\dots ,l_\z$ we set 
\[
f_i^\z :=  - h_i.
\]
For $i=1,\dots , n$ we set
\[
e_i := m_i + a(m_i),
\]
where $a(m_i)\in \a_\p$. This basis satisfies Equation \eqref{eq:new basis at origin}. To illustrate this we consider $\overline{f_i^\z}(o)$:
\[
\overline{f_i^\z}(o)=\overline{\phi(-h_i)}(o) = \overline{-h_i+n_i^\z + k_i^\z}(o) = \overline{n_i^\z}(o).
\]
Note that the two factors of $G/H_0\times \R^{l_\p}$ are in general not orthogonal with respect to the naturally reductive metric. We will now give an example.

\begin{ex}
In this example we will construct a 1-parameter family of naturally reductive structures on $SU(2)\times SU(2)\times \R^3$. Let $x_1,x_2,x_3$ be a basis of $\su(2)$ with
\[
[x_1,x_2] = -x_3,\quad [x_2,x_3]=-x_1,\quad [x_3,x_1]=-x_2.
\]
Let $g_{\su(2)}$ be a multiple of the killing form of $\su(2)$ which is positive definite and such that $x_1,x_2,x_3$ is an orthonormal basis with respect to $g_{\su(2)}$.
As base space consider the product naturally reductive space $SU(2)\times SU(2)$ with a flat naturally reductive structure. The reductive decomposition is:
\[
\g = \m =  \su(2) \oplus \su(2),
\]
with a metric $g_0 = (g_{\su(2)} \oplus \alpha g_{\su(2)})$ for some $\alpha> 0$. The Lie algebra $\s$ is in this case given by all derivations of $\g$, so $\s = \g$. As subalgebra $\k\cong \b\subset \s $ we take $\mf b := \{(m,m):m\in \su(2) \}$. We pick the following orthonormal basis with respect to $g_0$:
\[
m_i := (1+\alpha)^{-1/2}(x_i,x_i),\quad m_{i+3} :=(1+\alpha^{-1})^{-1/2}(x_i,-\frac{1}{\alpha}x_i)
\]
for $i=1,2,3$. The curvature is zero and the torsion is given by 
\[
T_0 = \frac{1}{\sqrt{1+\alpha}}(m_{123}+m_{156}+m_{264}+m_{345}) - cm_{456},
\]
where $c = (1-\alpha)\sqrt{1+\alpha^{-1}}/(1+\alpha)$. Now $m_1,m_2,m_3$ is a basis of $\mf b$. The $\ad(\b)$-invariant metrics on $\b$ are all of the form $B= \frac{1}{\lambda^2}g_0|_{\b\times \b}$ with $\lambda\in \R\backslash \{0\}$. An orthonormal basis of $\b$ with respect to $B$ is given by $b_i:=\lambda m_i$ for $i=1,2,3$. We have that
\begin{align*}
\varphi(k_1)&= \ad(b_1)=\frac{-\lambda}{\sqrt{1+\alpha}} (m_{23}+m_{56}),\\
\varphi(k_2)&= \ad(b_2)=\frac{-\lambda}{\sqrt{1+\alpha}} (m_{31}+m_{64}),\\
\varphi(k_3)&= \ad(b_3)=\frac{-\lambda}{\sqrt{1+\alpha}} (m_{12}+m_{45}).
\end{align*}
The Lie algebra $\g(\k)=\k\oplus \n \oplus \m$ is by Lemma \ref{lem:ideals in g(k)} isomorphic to $\g(\k) \cong \k \ltimes (\m \oplus \a )$. Lemma \ref{lem:Lie algebra hom f} gives a Lie algebra homomorphism $f:\g \to \m\oplus \a$. This gives us that $\g(\k)\cong \k\ltimes (f(\g)\oplus \a)$. The discussion above tells us that $(T,R)$ is always regular and that the connected Lie subgroup of  $f(\g)\oplus \a$ acts transitively on our space. Hence the naturally reductive space $G(\k)/H(\k)$ is isomorphic to the Lie group $SU(2)\times SU(2)\times \R^3$, where the Lie algebra of $\R^3$ is given by 
\[
\a = \mbox{span}\{f_1:=n_1+k_1,f_2:=n_2+k_2,f_3 :=n_3+k_3\}\subset \g(\k).
\]
We have
\[
f(m_i) = m_i - \sum_{j=1}^3\overline{g}(m_i,b_j)(n_j+k_j) = m_i -\lambda f_i, \quad f(m_{i+3}) = m_{i+3}.
\]
for $i=1,2,3$. Let 
\[
e_i:= m_i + \lambda f_i, \quad e_{i+3} :=m_{i+3},
\]
for $i=1,2,3$. Then $e_1,\dots ,e_6,f_1,f_2,f_3$ spans $\su(2)\oplus \su(2) \oplus \a$ and is an orthonormal basis for the naturally reductive metric. The torsion is given by
\[
T= T_0 +\frac{\lambda}{\sqrt{1+\alpha}} (f_1\wedge(e_{23}+e_{56})+f_2\wedge(e_{31}+e_{64})+f_3\wedge(e_{12}+e_{45})- 2f_{123})
\]
with
\[
T_0 = \frac{1}{\sqrt{1+\alpha}}(e_{123}+e_{156}+e_{264}+e_{345}) - ce_{456},
\]
where $c = (1-\alpha)\sqrt{1+\alpha^{-1}}/(1+\alpha)$. The curvature is given by
\[
R=\frac{\lambda^2}{1+\alpha} ((e_{23}+e_{56}+f_{23})^{\odot 2}+(e_{31}+e_{64}+f_{31})^{\odot 2}+(e_{12}+e_{45}+f_{12})^{\odot 2}).
\]
It is important to note that even though the homogeneous space is a product the constructed naturally reductive structures can not be written as products.
\end{ex}

\subsection{Base space $\R^n$} \label{subsec:Base space Rn}
As base space we take the symmetric Euclidean space $\R^n$. The spaces we obtain from our construction are the 2-step nilpotent naturally reductive spaces. These were first described in \cite{Gordon1985}. As naturally reductive decomposition of the base space we have
\[
\g = \h \oplus \m = \R^n,
\]
where $\h=\{0\}$ and $\g=\m=\R^n$ is an Abelian Lie algebra, thus $(T_0,R_0) = (0,0)$. The algebra $\s$ is given by
\[
\s = \so(n).
\]
Let $\k\subset \s$ be a subalgebra. The torsion and curvature are given by
\[
T=\sum_{i=1}^{l}\varphi(k_i)\wedge n_i+2T_{\n}\quad \mbox{and}\quad
R=\sum_{i=1}^{l} \psi(k_i)\odot \psi(k_i).
\]
The Lie algebra 
\[
\g(\k) = \R^n(\k) = \k \oplus \n \oplus \R^n
\]
has by Lemma \ref{lem:ideals in g(k)} the following ideal
\[
\mf l := \mf a \oplus \R^n.
\]
Since $T_0=0$ we have for $m_1,~m_2\in \R^n$ that $[m_1,m_2]\in \mf a$ and by Lemma \ref{lem:ideals in g(k)} $\mf a$ commutes with $\R^n$. Hence $\mf l$ is a 2-step nilpotent Lie algebra. This gives us a naturally reductive structure on the 2-step nilpotent Lie group $L$, where $L$ is the simply connected Lie group of $\mf l$. In particular the infinitesimal model $(T,R) $ is always regular in this case.
In this case it is much easier to describe an orthonormal basis of the tangent space at the origin. Let $n_1,\dots ,n_l$ be an orthonormal basis of $\n$ and $m_1,\dots,m_n$ is an orthonormal basis of $\R^n$. We have that $f_i=n_i+k_i\in \mf a$ for $i=1,\dots,l$ and $m_1,\dots,m_n\in \R^n$ form basis of $\mf l$ such that the corresponding fundamental vector fields span an orthonormal basis of the tangent space of the origin. In this basis the formula for the torsion and curvature are just given by the formulas above.
To illustrate how this works we give a concrete example.

\begin{ex}\label{ex:quaternionic heisenberg}
As base space start with $\R^4$ and let $\k = \su(2)\subset \so(4)$ be the subalgebra corresponding to the standard representation of $\su(2)$ on $\C^2\cong \R^4$. The $\ad(\k)$-invariant metric $B$ on $\k$ has to be a negative multiple of the Killing form. In a natural basis we get
\begin{align*}
\varphi(k_1) & = \lambda(e_{13}+e_{24}),\\
\varphi(k_2) & = \lambda(-e_{12}+e_{34}),\\
\varphi(k_3) & = \lambda(-e_{14}+e_{23}),
\end{align*}
where $k_1,k_2,k_3$ is an orthonormal basis with with respect to $B$. The formula for the torsion is
\[
T = \lambda(e_{13}+e_{24})\wedge f_1 + \lambda(-e_{12}+e_{34})\wedge f_2 + \lambda(-e_{14}+e_{23})\wedge f_3 + 4 \lambda f_{123}.
\]
and the curvature is
\[
R = (2\lambda f_{23}+\lambda(e_{13}+e_{24}))^{\odot 2} + (2\lambda f_{31}+\lambda(-e_{12}+e_{34}))^{\odot 2} + (2\lambda f_{12} + \lambda(-e_{14}+e_{23}))^{\odot 2},
\]
The underlying homogeneous space is the 2-step Nilpotent Lie group known as the $7$-dimensional quaternionic Heisenberg group. We will denote it by $QH^7$.
\end{ex}

For a Lie subalgebra $\k\subset \so(n)$ we will denote the 2-step nilpotent Lie algebra $\mf l$  obtained above by $\mf{nil}(\k)$ and the corresponding simply connected Lie group by $Nil(\k)$. 

\subsection{Base space of mixed type \label{subs:mixed base}}
In this subsection the base space is a product of the base spaces considered in the last two subsections:
\[
G/H\times \R^k.
\]
Just as before we suppose that $\g$ is semisimple and that $G/H$ is naturally reductive with respect to the canonical connection of the reductive decomposition $\g = \h \oplus \m$ and that $\im(R)=\ad(\h)$. Let $\overline{g}$ be the $\ad(\g)$-invariant metric on $\g$ from Subsection \ref{subsec:g semisimple}. We want to describe the naturally reductive spaces which are constructed from derivations of $\g' :=\g\oplus _{L.a.} \R^k$, where $\oplus_{L.a.}$ denotes a direct sum of Lie algebras. 
Any such derivation preserves the direct sum. We have that
\[
\mf s(\g') = \mf s(\g) \oplus \so(k).
\]
Let $\k\subset \mf s(\g')$ be a subalgebra with an $\ad(\k)$-invariant metric $B$. Let $\varphi_1:\k\to \so(\m)$ and $\varphi_2:\k \to \so(\R^k)$ be the Lie algebra representations defining $\k$. Let $\k_1:= \ker(\varphi_2)$, $\k_3:= \ker(\varphi_1)$ and let $\k_2$ be the orthogonal complement of $\k_1\oplus \k_3$ in $\k$ with respect to $B$. The projection $\pi:s(\g')\to \s(\g)$ restricted to $\k_1 \oplus \k_2$ is injective. We use $\pi$ to identify $\k_1\oplus \k_2$ with a subalgebra $\b_1\oplus \b_2$ of $\s(\g) = \z \oplus \p$, where $\pi(\k_i)=\b_i$ for $i=1,2$. As in Subsection \ref{subsec:g semisimple} we let $\b_{1,\p} := \b_1 \cap \p$ and $\b_{1,\z} := \b_{1,\p}^\perp$, where the orthogonal complement is taken in $\b_1$ with respect to $B$. In total we now have the following decomposition into ideals
\[
\k = \k_{1,\z}\oplus \k_{1,\p} \oplus \k_2 \oplus \k_3,
\]
with $\k_1 = \k_{1,\z}\oplus \k_{1,\p}$. We let the following be orthonormal bases with respect to $B$:
\[
\mbox{Span}\{b_1^{1,\z},\dots,b_{l_{1,\z}}^{1,\z}\}= \b_{1,\z},\quad
\mbox{Span}\{b_1^{1,\p},\dots,b_{l_{1,\p}}^{1,\p}\}=\b_{1,\p},\quad  \mbox{Span}\{b_1^2,\dots,b_{l_{2}}^{2}\}= \b_{2}.
\]
As before let
\begin{align*}
a_1 &: \g \to \g'(\k);\quad a(x) := \sum_{i=1}^{l_1} \overline{g}(x,b^1_i) (n^1_i+k^1_i),\\
a_2 &: \g \to \g'(\k);\quad a(x) := \sum_{i=1}^{l_2} \overline{g}(x,b^2_i) (n^2_i+k^2_i),
\end{align*}
where $b^j_1,\dots,b^j_{l_j}$ is an orthonormal basis of $\b_j$ with respect to $B$ and $n^j_i,k^j_i,$ are the corresponding bases of $\n^j$ and $\k^j$, respectively. Lastly define 
\[
a:=a_1+a_2:\g \to \g'(\k).
\]
We get analogous to Lemma \ref{lem:Lie algebra hom f}  that

\begin{lem}
The linear map $f:\g \to \g'(\k)$ defined by 
\[
f(x) = x - a(x)
\]
is an injective Lie algebra homomorphism.
\end{lem}

The direct sum $\k = \k_{1,\z}\oplus \k_{1,\p}\oplus \k_2 \oplus \k_3$ gives us also a corresponding direct sum of the diagonal $\a = \a_{1,\z} \oplus \a_{1,\p}\oplus \a_2 \oplus \a _3 \subset \k\oplus \n$.

\begin{lem}\label{lem:transitive action}
The algebras $\mf{nil}(\k_2\oplus \k_3)=\mf a_2\oplus \mf a_3 \oplus \R^k$ and $\a_{1,\p}$ are subalgebras of $\g'(\k)$, they commute with each other and with $f(\g)$. Furthermore they both have zero intersection with $f(\g)$, i.e.:
\[
f(\g)\oplus_{L.a.} \mf{nil}(\k_2\oplus \k_3) \oplus_{L.a.} \a_{1,\p} \subset \g'(\k).
\]
Moreover the projection of $f(\g)\oplus \mf{nil}(\k_2\oplus \k_3) \oplus \a_{1,\p}$ along $\h\oplus \k$ onto $\n\oplus \m\oplus \R^k\subset \g'(\k)$ is surjective.
\end{lem}
\begin{proof}
Straight forward check.
\end{proof}
If $H'(\k)\subset G'(\k)$ is closed then the connected subgroup of  $f(\g)\oplus_{L.a} \mf{nil}(\k_2\oplus \k_3) \oplus_{L.a.} \a_{1,\p}$ acts transitively on the homogeneous space $G'(\k)/H'(\k)$ by the above lemma. As before let 
\[ 
\phi:= f\oplus i :\g\oplus \mf{nil}(\k_2\oplus \k_3)\oplus \a_{1,\p}\to \g'(\k),
\]
where $i:\mf{nil}(\k_2\oplus \k_3)\oplus \a_{1,\p}\to \g'(\k)$ is the inclusion. Let $G \times Nil(\k_2\oplus \k_3)\times A_{1,\p}$ be the simply connected Lie group with Lie algebra $\g\oplus \mf{nil}(\k_2\oplus \k_3) \oplus \a_{1,\p}$. Let 
\[
\Phi:G \times Nil(\k_2\oplus \k_3)\times A_{1,\p}\to G'(\k)
\]
be the Lie group homomorphism with its derivative at the origin given by $\phi$. The isotropy algebra of $G \times Nil(\k_2\oplus \k_3)\times A_{1,\p}$ is given by 
\[
\h':=\phi^{-1}(\h\oplus \k) = \{h+a_2(h):h\in \ker(a_1|_\h)\}.
\]
Let $\h_0:=\ker(a|_\h)\subset \h'$. Note that $\h_0$ is an ideal in $\h$ just as in the previous subsection. Hence $\h_0$ is also an ideal in $\h'$. Let $\h_2$ be a complementary ideal of $\h_0$ in $\h'$. Lastly let $\h_1$ be a complementary ideal of $\ker(a_1|_\h)$ in $\h$, this exists because $\h$ is compact and $\ker(a_1|_\h)$ is an ideal in $\h$ by the same argument as for $\h_0$.
Then $a_1|_{\h_1}:\h_1 \to \a_{1,\z}$ is a bijection. As before let $\{h_1,\dots ,h_{l_{1,z}}\}$ be a basis of $\h_1$ such that $\overline{g}(h_i,b^{1,\z}_j)  = \delta_{ij}$. 

Let $H_0$ be the connected subgroup of $G$ with Lie subalgebra $\h_0$ and let $H_2$ be the connected subgroup of $G\times Nil(\k_2\oplus \k_3)$ with Lie subalgebra $\h_2$. Since the projection of $\h_2$ to $\a_2$ is injective and $Nil(\k_2\oplus \k_3)$ is simply connected the subgroup $H_2\subset G\times Nil(\k_2\oplus \k_3)$ is always a closed subgroup and is isomorphic to $\R^q$, with $q=\dim(\h_2)$.

We get that
\begin{equation}
G'(\k)/H'(\k) \cong (G \times Nil(\k_2\oplus \k_3)\times A_{1,\p})/ (H_0\times H_2).\label{eq:mixed type hom space}
\end{equation}
The new infinitesimal model is regular precisely when $H_0$ is closed in $G$. Now we describe the naturally reductive structure on the following reductive decomposition associated to \eqref{eq:mixed type hom space}:
\[
\h_0\oplus \h_2 \oplus \h_1 \oplus \m \oplus \a_2 \oplus \a_3 \oplus \R^k\oplus \a_{1,\p}\subset \g'(\k),
\]
where $\h_0\oplus \h_2$ is the isotropy algebra. We have an orthonormal basis $n_1^{1,\z},\dots,n_{l_{1,\z}}^{1,z}$ of $\n_{1,z}$, and $n_1^{1,\p},\dots ,n_{l_{1,\p}}^{1,\p}$ of $\n_{1,\p}$, and $n_1^2,\dots ,n_{l_2}^2$ of $\n_2$, and $n_1^3,\dots ,n_{l_3}^3$ of $\n_3$. Let $m_1,\dots,m_n$ be an orthonormal basis of $\m$ and let $m_{n+1},\dots,m_{n+k}$ be an orthonormal basis of $\R^k$. We know the formula for the torsion and curvature in this basis. Hence all we have to do is to find a basis 
\[
f_1^{1,\z},\dots,f_{l_{1,\z}}^{1,\z},f_1^{1,\p},\dots,f_{l_{1,\p}}^{1,\p},f_1^2,\dots, f_{l_2}^2,f_1^3,\dots, f_{l_3}^3, e_1,\dots,e_{n+k}
\]
of 
\[
\h_1 \oplus \m \oplus \a_2 \oplus \a_3 \oplus \R^k\oplus \a_{1,\p},
\]
such that for the fundamental vector fields the following holds: 
\begin{align}
\overline{f_j^{1,\z}}(o) = \overline{n_j^{1,\z}}(o),\quad \overline{f_j^{1,\p}}(o) = \overline{n_j^{1,\p}}(o),\quad \overline{f_j^2}(o) = \overline{n_j^2}(o), \quad\overline{f_j^3}(o) = \overline{n_j^3}(o),\quad \overline{e_j}(o) = \overline{m_j}(o), \label{eq:new basis at origin for mixed type}
\end{align}
where $o$ is the chosen origin. Such a basis will give an orthonormal basis of the tangent space at the origin and thus describe the left invariant metric. Also in this basis the formula for the torsion and curvature of the naturally reductive connection  are just \eqref{eq: T} and \eqref{eq: R}, respectively. 
For $i=1,\dots, l_{1,\z}$ we set
\[
f_i^{1,\z} := -h_i-a_2(h_i).
\]
For $i=1,\dots, l_2$ we set
\[
f^2_i:= n^2_i + k^2_i\in \a_2.
\]
For $i=1,\dots ,l_3$ we set
\[
f^3_i := n^3_i+k_i^3 \in \a_3.
\]
For $i=1,\dots, l_{1,\p}$ we set
\[
f_i^{1,\p} := n_i^{1,\p} + k_i^{1,\p}\in \a_{1,\p}.
\]
For $i=1,\dots , n$ we set
\[
e_i := m_i + a(m_i).
\]
For $i=n+1,\dots,n+k$ we set
\[
e_i := m_i \in \R^k.
\]
Just as in the case in Subsection \ref{subsec:g semisimple} this basis satisfies \eqref{eq:new basis at origin for mixed type}. Again as illustration let us consider the $\overline{f^{1,\z}_i}(o)$:
\[
\overline{f^{1,\z}_i}(o) = \overline{\phi(-h_i - a_2(h_i))}(o)=\overline{-h_i + a_1(h_i)}(o) = \overline{n_i^{1,\z}}(o).
\]

A simple example of one of these spaces appears in \cite{KowalskiVanhecke1985}. Here they describe the spaces $(SU(2)\times H^3)/\R$, where $H^3$ is a simply connected 3-dimensional Heisenberg group. In the following example we show how we can obtain these spaces from our construction.

\begin{ex}
As base space we take $SU(2)/SO(2)\times \R^2$. The first factor is the symmetric space $S^2$. The infinitesimal model is up to a rescaling given by $(0,R= - e_{12}\odot e_{12})$. This gives a reductive decomposition of $\su(2) = \h\oplus \m$, with $\h=\mbox{span}\{h=e_{12}\}$ and $\m = \mbox{span}\{e_1,e_2\}$. 
The Lie algebra $\s$ is given by $\s= \mbox{span}\{h\}\oplus \so(2) \cong \so(2)\oplus \so(2)$. We pick an element $k=(\lambda h, \mu h)\in \s$, with $\lambda\neq 0$, $\mu\neq 0$. Let $\k$ be the 1-dimensional Lie algebra spanned by $k$ and with a metric such that $k$ has norm 1. We have that
\[
\varphi(k) = \lambda e_{12} + \mu e_{34}.
\]
In this case $\k = \k_2$. By Lemma \ref{lem:transitive action} the Lie subgroup $SU(2)\times H^3$ of $G'(\k)$ acts transitively by isometries, where $H^3$ is the simply connected 3-dimensional Heisenberg group. Let $\h^3$ be the Lie algebra of $H^3$ and let $e_3,e_4,f_1$ be a basis with  as only non-vanishing Lie bracket
\[
[e_3,e_4] = -\mu f_1.
\]
The isotropy algebra is given by 
\[
\h_2 = \mbox{Span}\{h + \lambda f_1\}.
\]
As reductive complement of $\h_2$ in $\su(2)\oplus \h^3$ we have $\n\oplus \m = \mbox{span}\{f_1,e_1,e_2,e_3,e_4\}$. From the description above $\{f_1,e_1,e_2,e_3,e_4\}$ is an orthonormal basis and the torsion and curvature in this basis are given by
\[
T = f_1 \wedge( \lambda e_{12} + \mu e_{34} ),\quad \mbox{and}\quad R =-e_{12}\odot e_{12} + ( \lambda e_{12} + \mu e_{34} )\odot ( \lambda e_{12} + \mu e_{34} ).
\]
The isotropy group is the connected Lie subgroup with Lie subalgebra $\h_2$. It is isomorphic to $\R$ and it is closed for all $\lambda,\mu\in \R$. So we have a two-parameter family of naturally reductive structures on 
\[
(SU(2)\times H^3)/\R. 
\] 
\end{ex}

We give an example which is a bit more interesting and illustrates better how the newly constructed spaces can look like.

\begin{ex}\label{ex:su(3)/s1 times R4}
As base space we take the product of a particular Aloff-Wallach space $SU(3)/S^1$ with $\R^4$. The isotropy algebra of $SU(3)/S^1$ is spanned by
$\h:=\mbox{Lie}(S^1)=\mbox{span} \left\{ h:= 
\left(\begin{array}{c c c}
i & 0 & 0\\
0 & i & 0\\
0 & 0 & -2i
\end{array}\right)\right\}.$
We pick the following orthonormal basis for the orthogonal complement of $\h$ with respect to the Killing form:
\begin{align*}
m_1&:= 
\left(\begin{array}{c c c}
0 & -1 & 0\\
1 & 0 & 0\\
0 & 0 & 0
\end{array}\right), \quad m_2:= 
\left(\begin{array}{c c c}
i & 0 & 0\\
0 & -i & 0\\
0 & 0 & 0
\end{array}\right),\quad m_3:= 
\left(\begin{array}{c c c}
0 & i & 0\\
i & 0 & 0\\
0 & 0 & 0
\end{array}\right),\\
m_4&:= 
\left(\begin{array}{c c c}
0 & 0 & -1\\
0 & 0 & 0\\
1 & 0 & 0
\end{array}\right),\quad m_5:= 
\left(\begin{array}{c c c}
0 & 0 & i\\
0 & 0 & 0\\
i & 0 & 0
\end{array}\right),\quad m_6:= 
\left(\begin{array}{c c c}
0 & 0 & 0\\
0 & 0 & -1\\
0 & 1 & 0
\end{array}\right),\quad m_7:= 
\left(\begin{array}{c c c}
0 & 0 & 0\\
0 & 0 & i\\
0 & i & 0
\end{array}\right).
\end{align*}
The torsion is given by
\[
T_0 = -2 m_{123} - m_1\wedge (m_{46}+m_{57}) + m_2\wedge (m_{45}-m_{67} ) + m_3\wedge (m_{47}-m_{56}).
\]
The curvature is given by
\[
R_0 = -\ad(h)^{\odot 2} = -( -3m_{45} -3 m_{67})^{\odot 2}.
\]
The $S^1$-invariant vectors are spanned by $m_1,m_2$ and $m_3$. Lastly let $m_8,m_9,m_{10},m_{11}$ be an orthonormal basis of $\R^4$. The Lie algebra $\s$ is given by
\[
\mbox{span}\{h,m_1,m_2,m_3\} \oplus \so(4) \cong \mf{u}(2)\oplus \so(4).
\]
We define the subalgebra $\k\subset \s$ by
\begin{align*}
 \varphi(k_1)&:= \lambda(\ad(m_1)+ m_{8,10}+m_{9,11}),\\
 \varphi(k_2)&:= \lambda(\ad(m_2) - m_{8,9}+m_{10,11}),\\
 \varphi(k_3)&:= \lambda(\ad(m_3)- m_{8,11}+m_{9,10}),\\ 
 \varphi(k_4)&:= \mu_1 \ad(h) + \mu_2 (m_{8,9} + m_{10,11}),
\end{align*}
where $k_1,k_2,k_3,k_4$ is an orthonormal basis of $\k$. We consider the case that $\mu_2=0$. Then $\k_1 = \k_{1,\z} = \mbox{Span}\{k_4\}$ and $\k_2 = \mbox{Span}\{k_1,k_2,k_3\}$. From the discussion above we know that the new infinitesimal model $(T,R)$ constructed from $\k$ is always regular. The connected subgroup of $G'(\k)$ with Lie subalgebra $f(\g)\oplus \mf{nil}(\k_2)$ acts transitively on $G'(\k)/H'(\k)$ and the isotropy algebra is trivial. The simply connected Lie group with Lie algebra $\g\oplus \mf{nil}(\k_2)$ is $SU(3)\times QH^7$. We describe the naturally reductive structure directly on $SU(3)\times QH^7$. As basis for the Lie algebra of $QH^7$ we take
\[
m_8,m_9,m_{10},m_{11},f_1,f_2,f_3,
\]
where $m_i$ corresponds to $e_{i-7}$ from Example \ref{ex:quaternionic heisenberg}. An orthonormal basis at the origin is given by
\[
\{f_1,f_2,f_3,f_4 = - \frac{1}{3\mu_1} h,e_1 := m_1 + \lambda f_1,~e_2 := m_2 +\lambda f_2,~e_3: = m_3 +\lambda f_3,e_j=m_j\}
\]
for $j=4,\dots ,11$. The torsion in this basis is given by

\begin{align*}
T &= -2 e_{123} - e_1\wedge (e_{57}+e_{46}) + e_2\wedge (e_{45}-e_{67} ) + e_3\wedge (e_{47}-e_{56}) \\
 &+ \lambda f_1 \wedge (2e_{23}+e_{57}+e_{46}+e_{8,10}+e_{8,11}) + \lambda f_2 \wedge (2e_{31} - e_{45}+e_{67} +e_{89} -e_{10,11}) \\
  &+ \lambda f_3\wedge (2e_{12} -e_{47}+e_{56}+e_{8,11}+e_{10,11}) -  3\mu_1 f_4 \wedge (e_{45}+e_{67} ) + 4 \lambda f_{123}.
\end{align*}
The curvature is given by
\begin{align*}
R = &\lambda^2(2e_{23}+e_{57}+e_{46}+e_{8,10}+e_{8,11}+ 2f_{23})^{\odot 2} +  \lambda^2 (2e_{31} - e_{45}+e_{67} +e_{89} -e_{10,11}+ 2f_{31})^{\odot 2} +\\
&\lambda^2 (2e_{12} -e_{47}+e_{56}+e_{8,11}+e_{10,11} + 2f_{12})^{\odot 2} +
(3\mu_1)^2 (e_{34}+e_{67})^{\odot 2}.
\end{align*}
The base space of this example illustrates nicely that there are multiple different choices for the Lie algebra $\k$. We decided for simplicity to pick $\mu_2=0$. When $\mu_2\neq 0$ we simply obtain a new naturally reductive space. An other option is to pick $\mu_1=0$ and $\mu_2 =\mu$ and then $\varphi(k_5) = \mu(m_{8,10} - m_{9,11})$ and $\varphi(k_6)= \mu(m_{8,11} + m_{9,10})$. In this case $\k = \k_2 \oplus \k_3$, with $\k_2 = \Span\{ k_1,k_2,k_3\}\cong \su(2)$, $\k_3 = \Span\{k_4,k_5,k_6\}\cong \su(2)$. We readily see from or previous discussion that the space is isomorphic to
\[
SU(3)/S^1 \times Nil(\so(4)),
\]
where $Nil(\mf{so}(4))$ is a 10-dimensional 2-step Nilpotent Lie group described in Subsection \ref{subsec:Base space Rn}. In all cases the naturally reductive structure is not a product structure. 
\end{ex}

\bibliographystyle{abbrv}
\bibliography{../NaturalReductiveDim7}

\end{document}